\documentclass[11pt]{amsart}
\usepackage{amscd,amssymb}
\usepackage{amsthm,amsmath,amssymb}
\usepackage[matrix,arrow]{xy}

\newtheorem{theorem}[equation]{Theorem}

\newtheorem{lemma}[equation]{Lemma}
\newtheorem{corollary}[equation]{Corollary}

\theoremstyle{definition}

\newtheorem{definition}[equation]{Definition}
\newtheorem{remark}[equation]{Remark}
\theoremstyle{remark}

\makeatletter\@addtoreset{equation}{section} \makeatother

\newfont{\FieldFont}{msbm10 scaled\magstep1}

\begin{document}

\Large
\begin{center}
\textbf{An Omori-Yau maximum principle for semi-elliptic operators and
Liouville-type theorems}
\end{center}
\vspace{4mm}

\normalsize
\begin{center}
\textbf{Kyusik Hong and Chanyoung Sung
}\\
\vspace{5mm} \small{\itshape Department of Mathematics, Konkuk
University, Seoul, $143$-$701$, Republic of Korea.}
\end{center}

\hrulefill \vspace{2mm}

 \small {\textbf{Abstract.} We generalize the Omori-Yau
 almost maximum principle of the Laplace-Beltrami operator on a complete Riemannian manifold $M$ to a
 second-order linear semi-elliptic operator $L$ with bounded coefficients and no zeroth order term.

Using this result, we prove some Liouville-type theorems for a
real-valued $C^{2}$ function $f$ on $M$  satisfying
$L f \geq F(f)+ H(|\nabla f|) $ for  real-valued continuous
functions $F$  and $H$ on $\Bbb R$ such that $H(0)=0$.

 \vspace{2mm}

 }

\hrulefill \normalsize \hrulefill \normalsize

\footnotetext[1]{Date : \today. }

\footnotetext[2]{E-mail addresses : kszoo@postech.ac.kr,
cysung@kias.re.kr}

\thispagestyle{empty}

\section{Introduction}

Let $(M,g)$ be a smooth complete Riemannian manifold of dimension
$n$. A second-order linear differential operator
$L: C^\infty(M)\rightarrow C^\infty(M)$ without
zeroth order term can be written as
\begin{equation}\label{eq111}
Lf=Tr(A\circ hess(f))+g(V,\nabla f),
\end{equation}
where $A\in \Gamma(\textrm{End}(TM))$ is self-adjoint with respect
to $g$,  $hess(f)\in \Gamma(\textrm{End} (TM))$ is the Hessian of
$f$ in the form defined by $hess(f)(X)=\nabla_X\nabla f$ for $X\in
\Gamma(TM)$, and finally $V\in \Gamma(TM)$. In this article, we will
deal with the semi-elliptic case, i.e. $A$ is positive semi-definite
at each point, and we always assume that
\begin{equation}\label{222}
\sup_MTr(A)+\sup_M|V|< \infty.
\end{equation}
The purpose of this paper is to show that such a operator $L$ shares
important properties with the Laplace-Beltrami operator $\Delta$,
particularly Omori-Yau almost maximum principle and Liouville-type
theorems for subharmonic functions.

To state our main theorem, we need the following definitions.
\begin{definition}
Let $u$ be a real-valued continuous function on $M$ and let a point
$p \in M$.

\begin{itemize}
\item  a function $u$ is called proper, if the set $\{p:u(p)\leq r\}$ is compact for
every real number $r$.

\item a function $v$ defined on a neighborhood $U_{p}$
of p is called an upper-supporting function for $u$ at p, if the
conditions $v(p)=u(p)$ and $v \geq u$ hold in $U_{p}$.

\end{itemize}
\end{definition}
\begin{definition}
A proper continuous function $u: M \rightarrow \mathbb{R}$ is called
an $L$-tamed exhaustion, if the following condition
holds:
\begin{enumerate}
\item $u \geq 0$.

\item At all points $p \in M$ it has a $C^2$-smooth,
upper-supporting function $v$ at $p$ defined on an open neighborhood
$U_{p}$ such that $|\nabla v|_{p} |\leq 1 $ and $L
v|_{p} \leq 1 $.
\end{enumerate}

\end{definition}
Once there is an $L$-tamed exhaustion, it easily follows that another $L$-tamed exhaustion can be chosen so that its local upper-supporting functions $v$ satisfy  $|\nabla v| \leq 1 $ and $Lv \leq 1 $ not just at one point $p$ but also on its whole $U_p$.\footnote{For example, given an $L$-tamed exhaustion $u$, one can just take $\frac{u}{2}$. Then at each point $p$, $\frac{v}{2}$ can be used for an upper-supporting function which satisfies that $|\nabla \frac{v}{2}|_p| \leq \frac{1}{2} $ and $L(\frac{v}{2})|_p \leq \frac{1}{2}$. Thus by taking $U_p$ smaller (if necessary), one can achieve $|\nabla \frac{v}{2}| \leq 1$ and $L(\frac{v}{2}) \leq 1$ on each $U_p$.}
The existence of an $L$-tamed exhaustion function on a complete
Riemannian manifold is guaranteed if certain curvature conditions are satisfied. For instance,
\begin{theorem}(H.L. Royden \cite[Proposition 2]{Royden})
Every complete Riemannian manifold with its sectional curvature bounded below admits an $L$-tamed exhaustion function.
\end{theorem}

When $L$ is the Laplace-Beltrami operator $\Delta$, a stronger result holds.
K.-T. Kim and H. Lee \cite{ktk07} have shown that a $\Delta$-tamed exhaustion function exists  if the Ricci
curvature $\textrm{Ric}$ satisfies
\begin{eqnarray}\label{bless}
Ric(\nabla r, \nabla r)\geq -B\rho(r)
\end{eqnarray}
for some constant $B>0$, where $r$ is the distance from an arbitrarily
fixed point in $M$ and a smooth nondecreasing function $\rho(r)$ on $[0,\infty)$ satisfies
\begin{eqnarray}\label{you}
\rho(0)=1,\ \ \ \ \int_0^\infty\frac{1}{\sqrt{\rho(t)}}\ dt=\infty,
\end{eqnarray}
\begin{eqnarray}\label{happy}
\rho^{(2k+1)}(0)=0 \ \ \forall k\geq 0,\ \ \ \ \limsup_{t\rightarrow \infty} \frac{t\rho(\sqrt{t})}{\rho(t)} < \infty.
\end{eqnarray}
For example, if
$$\textrm{Ric}(\nabla r,\nabla r) \geq -B\ r^2(\log
r)^2(\log(\log r))^2\cdots (\log^{k}r)^2$$ for $r\gg 1$, a $\Delta$-tamed exhaustion always exists.

A. Ratto, M. Rigoli, and A. Setti \cite{rrs} showed that if the above Ricci curvature condition (\ref{bless}) holds, then
for every real-valued $C^2$ function $f$ on $M$ which is bounded above,
there exists a sequence $\{p_{k}\}$ on $M$ such that
$$\lim_{k\rightarrow \infty}|\nabla f(p_{k})|=0,\
\limsup_{k\rightarrow \infty}{\Delta}f(p_{k})\leq 0, ~\textrm{and}~
\lim_{k\rightarrow \infty}f(p_{k})=\sup_{M}f.$$ This property is the
well-known Omori-Yau almost maximum principle for the Laplacian,
which was first proven by H. Omori \cite{omori} and S.T. Yau
\cite{Yau} when the Ricci curvature is only bounded below. K.-T. Kim
and H. Lee \cite{ktk07} showed the above maximum principle holds
whenever there exists a $\Delta$-tamed exhaustion. We will prove the analogous
maximum principle for the above semi-elliptic
operator $L$ also holds whenever there exists an
$L$-tamed exhaustion  by following their method in
\cite{ktk07}.

L.J. Alias, D. Impera, and M. Rigoli also proved a generalized
Omori-Yau maximum principle for $L$, when  the sectional
curvature $K$ satisfies
\begin{eqnarray}\label{AIR}
K(\Sigma)\geq -B\rho(r)
\end{eqnarray}
for any tangent 2-plane $\Sigma$ containing $\nabla r$, where $B>0$ is a constant, $r$ is the distance from an arbitrarily
fixed point in $M$, and $\rho(r)$ is as in (\ref{you}, \ref{happy}). (For a proof, see
\cite[Corollary 3]{alias} which is actually stated for $L$ with no first order terms but can be trivially extended to the general $L$.) It remains as a natural question whether the condition (\ref{AIR}) implies the existence of an $L$-tamed exhaustion.

Recently A. Borb\'ely \cite{bor} proved that the
Omori-Yau maximum principle for $\Delta$ holds without (\ref{happy}) in
Ratto-Rigoli-Setti's condition.
The relation between A. Borb\'ely's condition and the existence of
$\Delta$-tamed exhaustion also remains for further
study.

Now come applications. One of main applications of the Omori-Yau maximum principle is a
generalized Liouville-type theorem which gives a condition for the a
priori boundedness of solutions of Laplace-type differential
inequalities. This idea has originated from Cheng and Yau \cite{CY},
and been further extended by \cite{rrs}, \cite{suh}, \cite{cysung11},
etc. We can now extend the results of \cite{cysung11} to our semi-elliptic operator
$L$.
\begin{theorem}\label{main1}
Let $M$ be a smooth complete Riemannian manifold admitting an
$L$-tamed exhaustion function. Suppose that a $C^2$
function $f : M \rightarrow \mathbb{R}$  is bounded below and
satisfies $Lf \geq F(f)+ H(|\nabla f|)$
for real-valued continuous functions $F$ and $H$ on $\Bbb R$ such that $H(0)=0$.
\begin{enumerate}

\item If $\liminf_{x \rightarrow \infty}\frac{F(x)}{x^{\nu}}>0$ for
some $\nu>1,$ then $f$ is bounded such that $F(\sup f)\leq 0.$

\item If $\liminf_{x \rightarrow \infty}\frac{F(x)}{x^\nu}\leq 0$ for
any $\nu>1,$ then $\sup f= \infty $ or $f$ is bounded such that
$F(\sup f)\leq 0.$

\end{enumerate}

\end{theorem}

\begin{theorem}\label{main2}
Let $M$ be as in Theorem \ref{main1}. Suppose that a $C^2$ function
$f : M \rightarrow \mathbb{R}$ is bounded above and satisfies
$Lf \geq F(f)+ H(|\nabla f|)$ for $F$ and $H$ as in
the above theorem.
\begin{enumerate}
\item If $\liminf_{x \rightarrow -\infty}\frac{F(x)}{(-x)^{\nu}}>0$ for
some $\nu \leq 1,$ then $f$ is bounded such that $F(\inf f)\leq 0.$

\item If $\liminf_{x \rightarrow -\infty}\frac{F(x)}{(-x)^{\nu}}\leq 0$ for
any $\nu \leq 1,$ then $\inf f= -\infty $ or $f$ is bounded such that
$F(\inf f)\leq 0.$
\end{enumerate}

\end{theorem}

As a corollary, we give a semi-elliptic generalization of Liouville's theorem stating that any $f \in C^{2}(\mathbb{R}^{2})$ which is subharmonic $(\Delta f \geq 0)$ and bounded above must be constant.
\begin{corollary}\label{cor3}
Let $M$ be as in Theorem \ref{main1}.
\begin{enumerate}

\item There exits no $f \in C^{2}(M)$ which is bounded above and
$L f \geq c$ for a constant $c>0$.

\item Any $f \in C^{2}(M)$ which is non-positive and satisfies
$L f \geq c|f|^d$ for some positive constants $c$
and $d$ must be identically zero.
\end{enumerate}

\end{corollary}

\begin{remark}
Theorem \ref{main1} and \ref{main2} can be easily extended to any linear second-order semi-elliptic operator $$L
+h,$$ for $h\in C^\infty(M)$ just by considering $Lf\geq F(f)-hf+ H(|\nabla f|).$
\end{remark}

There are many other conditions under which the Omori-Yau maximum
principle holds, and also lots of Liouville-type
theorems for a variety of subharmonic functions. For instance, the
readers may be referred to \cite{karp1, karp, Leung, nad, prs2,
prs1, prs, take, yibing, xu}, and references therein.

\section{Generalized Omori-Yau maximum principle}

\begin{theorem}\label{main3}
Let $M$ be a smooth complete Riemannian $n$-manifold admitting an
$L$-tamed exhaustion function. Then for every
real-valued $C^2$ function $f$ on $M$ which is bounded above, there
exists a sequence $\{p_{k}\}$ on $M$ satisfying the following
properties:
$$\lim_{k\rightarrow \infty}|\nabla f(p_{k})|=0,\
\limsup_{k\rightarrow \infty}Lf(p_{k})\leq 0,
~\textrm{and}~ \lim_{k\rightarrow \infty}f(p_{k})=\sup_{M}f.$$
\end{theorem}

\begin{proof}
The proof is similar to the method in the article \cite{ktk07}, which uses a sequence of compact-supported approximations of $f$, which obviously attain their maximums.  Without loss of generality, we may assume that $\sup_{M}f>0$ by adding some
positive constant. Take an $L$-tamed exhaustion function $u$.

Now, we choose a point $p \in M$ such that $f(p)>0$. For each
$\epsilon>0$, let $$X_{\epsilon}=\{x \in M |
u(x)<\frac{1}{\epsilon}\}.$$ Then $X_{\epsilon}$ forms an increasing
sequence of open subsets of $M$ and each closure
$\overline{X}_{\epsilon}$ gives rise to a compact exhaustion of $M$
as $\epsilon \downarrow 0.$

Taking a positive constant $r$ such that $p \in X_{r}$. The
continuous function $$(1-ru(x))f(x)$$ vanishes on the boundary of
$X_{r}$. Since $(1-ru(p))f(p)>0$, and since $\overline{X}_{r}$ is
compact, the function $(1-ru)f$ attains its maximum value in the set
$X_{r}$, say at $p_{r} \in X_{r}$, respectively. It is obvious that
the maximum value is positive. From now on, we fix $r$.

Let $\epsilon$ be any positive constant smaller than $r$. Then $p
\in X_{r} \subset X_{\epsilon} $ and $$(1-\epsilon u(p))f(p)\geq
(1-ru(p))f(p)>0.$$ In the same way, the function $(1-\epsilon u)f$
attains a positive maximum value in the set $X_{\epsilon}$, say at
$p_{\epsilon} \in X_{\epsilon}.$

Since $A$, in the notation \eqref{eq111}, is symmetric, it is
diagonalizable at each point in an orthonormal basis, so we can take
a normal coordinate $(x_{1}, \cdots, x_{n})$ around $p_{\epsilon}
\in M$ such that $A$ at $p_{\epsilon}$ is represented as a diagonal
matrix, and hence
\begin{equation}\label{eq1}
L h|_{p_{\epsilon}}
=\sum_{l}a_{ll}(p_{\epsilon})\frac{\partial^2}{\partial
x_{l}^2}h|_{p_{\epsilon}}+\sum_{l}
a_{l}(p_{\epsilon})\frac{\partial}{\partial x_{l}}h
|_{p_{\epsilon}},
\end{equation}
for a real-valued function $h$ on $M$, where each
$a_{ll}(p_{\epsilon})$ is nonnegative, and the entries
$a_{ll}(p_{\epsilon})$ and $|a_{l}(p_{\epsilon})|$ are bounded above
as $p_{\epsilon}$ varies by \eqref{222}. For a notational
convenience, let's introduce locally-defined differential operators
\begin{equation}\label{eq10}
\widetilde{\nabla}:=(a_{11}(p_{\epsilon})\frac{\partial}{\partial
x_{1}}, ~\cdots~, a_{nn}(p_{\epsilon})\frac{\partial}{\partial
x_{n}})~~\textrm{and}~~
\widetilde{\nabla}_{1}:=a_{1}(p_{\epsilon})\frac{\partial}{\partial
x_{1}}+ ~\cdots~+ a_{n}(p_{\epsilon})\frac{\partial}{\partial
x_{n}},
\end{equation}
and put $d_{l}=a_{ll}(p_{\epsilon})$ and
$e_{l}=\sqrt{n}|a_{l}(p_{\epsilon})|$ for $1 \leq l \leq n$.

If $h$ has an extremal value at point $p_{\epsilon}$,
$$Lh |_{p_{\epsilon}}=\sum_{l}a_{ll}(p_{\epsilon})\frac{\partial^2}{\partial x_{l}^2}h|_{p_{\epsilon}}.$$

Furthermore, if a real-valued function AB on $M$ has an extremal
value at $p_{\epsilon}$, then one can obtain
\begin{equation}\label{eq2}
L(AB)|_{p_{\epsilon}}=(LA|_{p_{\epsilon}}-
\widetilde{\nabla}_{1}A|_{p_{\epsilon}})B(p_{\epsilon})+2\widetilde{\nabla}A|_{p_{\epsilon}}
\cdot \nabla B|_{p_{\epsilon}}
+A(p_{\epsilon})(LB|_{p_{\epsilon}}-
\widetilde{\nabla}_{1}B|_{p_{\epsilon}}).
\end{equation}

Note that $\widetilde{\nabla}A|_{p_{\epsilon}} \cdot \nabla
B|_{p_{\epsilon}}= \nabla A|_{p_{\epsilon}} \cdot
\widetilde{\nabla} B|_{p_{\epsilon}}.$

We may assume that $d_{1}$ and $e_{1}$ are the largest of $\{d_{1},
\cdots ,d_{n}\}$ and $\{ e_{1} , \cdots , e_{n}\}$ respectively.
Consider a $C^2$ upper-supporting function $v : U \rightarrow
\mathbb{R}$ for $u$ at $p_{\epsilon}$, where $U$ is an open
neighborhood of $p_{\epsilon}.$ Then we get $$ |\widetilde{\nabla}
v|_{p_{\epsilon}} |\leq d_{1},~ |\widetilde{\nabla}_{1}
v|_{p_{\epsilon}} |\leq e_{1},~ \textrm{and}~ L
v|_{p_{\epsilon}} \leq 1.$$ By taking $U$ further small, we may
assume that $U \subset X_{\epsilon}$ and $f$ is positive on $U$,
since $f(p_\epsilon)>0$. For every $x \in U$,
$$(1-\epsilon v(x))f(x) \leq (1-\epsilon u(x))f(x) \leq (1-\epsilon
u (p_{\epsilon}))f(p_{\epsilon})=(1-\epsilon
v(p_{\epsilon}))f(p_{\epsilon}).$$ Since $p_{\epsilon}$ is a local
maximum point of $(1-\epsilon v)f$, we get
$$\nabla [(1-\epsilon v)f]|_{p_{\epsilon}}= \widetilde{\nabla} [(1-\epsilon v)f]|_{p_{\epsilon}}=
\widetilde{\nabla}_{1}[(1-\epsilon v)f]|_{p_{\epsilon}}=0 .$$ By a
simple calculation, we have
$$ (1-\epsilon v(p_{\epsilon}))| \widetilde{\nabla} f(p_{\epsilon}) |=
\epsilon  | \widetilde{\nabla} v(p_{\epsilon}) | f(p_{\epsilon})
\leq \epsilon d_{1}( \sup_{M}f).$$ From
$v(p_{\epsilon})=u(p_{\epsilon})$, we get
$$(1-\epsilon u(p_{\epsilon}))| \widetilde{\nabla} f(p_{\epsilon}) | \leq \epsilon d_{1}(\sup_{M}f).$$
Also, because $ X_{r} \subset X_{\epsilon}$, we have
$$(1-\epsilon u(p_{r}))f(p_{r}) \leq (1-\epsilon
u(p_{\epsilon}))f(p_{\epsilon}).$$ This implies that
\begin{eqnarray*}
(1-ru(p_{r}))f(p_{r})|\widetilde{\nabla}f(p_{\epsilon})|&\leq&(1-\epsilon
u(p_{r}))f(p_{r})|\widetilde{\nabla}f(p_{\epsilon}) | \leq
(1-\epsilon
u(p_{\epsilon}))f(p_{\epsilon})|\widetilde{\nabla}f(p_{\epsilon})
|\\&\leq& f(p_{\epsilon})\epsilon d_{1} (\sup_{M}f) \leq \epsilon
d_{1}(\sup_{M}f)^{2}.
\end{eqnarray*}
So, we conclude that
$$ |\widetilde{\nabla}f(p_{\epsilon}) | \leq  \epsilon \frac{
d_{1}(\sup_{M}f)^{2}}{(1-r u(p_{r})) f(p_{r})}.$$ Note that
$K:=\frac{d_{1}(\sup_{M}f)^{2}}{(1-ru(p_{r})) f(p_{r})}$ is a
positive constant independent of $\epsilon$ with $\epsilon<r$.
Therefore, we obtain $$\lim_{\epsilon\rightarrow
0}|\widetilde{\nabla} f (p_{\epsilon})|=0.$$ By the same method as
above, we have
$$ | \nabla f(p_{\epsilon}) | \leq  \epsilon \frac{
(\sup_{M}f)^{2}}{(1-r u(p_{r})) f(p_{r})} \ \ \ \ \textrm{and}\ \ \ \
|\widetilde{\nabla}_{1}f(p_{\epsilon}) | \leq  \epsilon \frac{
e_{1}(\sup_{M}f)^{2}}{(1-r u(p_{r})) f(p_{r})}$$ Therefore, we get
$$\lim_{\epsilon \rightarrow 0}|\nabla f (p_{\epsilon})|=0\  \  \  \
\textrm{and}\  \  \ \ \lim_{\epsilon \rightarrow
0}|\widetilde{\nabla}_{1} f (p_{\epsilon})|=0.$$ Now we prove
$$\limsup_{\epsilon \rightarrow
0}Lf(p_{\epsilon})\leq 0.$$ Since $p_{\epsilon}$ is
a local maximum point of $(1-\epsilon v)f$, we have
$L((1-\epsilon v )f) \leq 0$ at point
$p_{\epsilon}.$ Using the formula (\ref{eq2}),
\begin{eqnarray*}
[L((1-\epsilon v )f)]|_{p_{\epsilon}} &=&-\epsilon
Lv|_{p_{\epsilon}}f(p_{\epsilon})+\epsilon
\widetilde{\nabla}_{1} v|_{p_{\epsilon}}f(p_{\epsilon})-2\epsilon
\nabla v|_{p_{\epsilon}} \cdot \widetilde{\nabla} f|_{p_{\epsilon}}+
(1-\epsilon v(p_{\epsilon}))Lf|_{p_{\epsilon}}
\\
& &- (1-\epsilon v(p_{\epsilon}))
\widetilde{\nabla}_{1}f|_{p_{\epsilon}}\\ &\leq& 0.
\end{eqnarray*}
Hence
\begin{eqnarray*}
(1-\epsilon v(p_{\epsilon}))Lf|_{p_{\epsilon}}
&\leq& 2\epsilon \nabla v|_{p_{\epsilon}} \cdot \widetilde{\nabla}
f|_{p_{\epsilon}} +\epsilon
Lv|_{p_{\epsilon}}f(p_{\epsilon})-\epsilon
\widetilde{\nabla}_{1} v|_{p_{\epsilon}}f(p_{\epsilon})+(1-\epsilon
v(p_{\epsilon}))
\widetilde{\nabla}_{1}f|_{p_{\epsilon}}\\
&\leq& \epsilon (2|\widetilde{\nabla }f|_{p_{\epsilon}} |+
\sup_{M}f+ e_{1}\sup_{M}f)+|\widetilde{\nabla}_{1}f|_{p_{\epsilon}}|
(1-\epsilon v(p_{\epsilon}))\\
&\leq& \epsilon ( 2\epsilon K + \sup_{M}f+ e_{1}\sup_{M}f
)+|\widetilde{\nabla}_{1}f|_{p_{\epsilon}}| (1-\epsilon
v(p_{\epsilon})).
\end{eqnarray*}
Since $1-\epsilon u(p_{\epsilon})=1-\epsilon v(p_{\epsilon})>0,$
we get
$$Lf|_{p_{\epsilon}}\leq \epsilon \frac {(2
\epsilon K + \sup_{M}f + e_{1}\sup_{M}f)} {(1-\epsilon
u(p_{\epsilon}))}+|\widetilde{\nabla}_{1}f|_{p_{\epsilon}}|.$$ As
above, we obtain
\begin{eqnarray*}
Lf|_{p_{\epsilon}}&\leq& \epsilon \frac {(2
\epsilon K + \sup_{M}f + e_{1}\sup_{M}f)(\sup_{M}f)} {(1-\epsilon
u(p_{\epsilon}))f(p_{\epsilon})}+\epsilon \frac{
e_{1}(\sup_{M}f)^{2}}{(1-r u(p_{r})) f(p_{r})}\\ &\leq& \epsilon
\frac {(2 \epsilon K + \sup_{M}f + e_{1}\sup_{M}f)(\sup_{M}f)} {(1-r
u(p_{r}))f(p_{r})}+\epsilon \frac{ e_{1}(\sup_{M}f)^{2}}{(1-r
u(p_{r})) f(p_{r})}.
\end{eqnarray*}
Therefore, we conclude that there is a positive constant $C$
independent of $\epsilon$ such that $Lf|_{p_{\epsilon}} \leq C \epsilon.$

It only remains to show that $\lim_{\epsilon\rightarrow
0}f(p_{\epsilon})=\sup_{M} f.$

Let $\eta$ be any positive constant such that $\sup_{M} f>\eta $. We
may choose a point $q \in M$ such that
$f(q)>\sup_{M}f-\frac{\eta}{2}$. Also choosing a positive constant
$\epsilon$ with $\epsilon < r$ such that $q\in X_{\epsilon}$ and
$\epsilon u(q) f(q) \leq \frac{\eta}{2}$ , we get
$$(1-\epsilon u(p_{\epsilon}))f(p_{\epsilon}) \geq
(1-\epsilon u(q))f(q) \geq \sup_{M}f-\eta.$$ Since $0<1-\epsilon
u(p_{\epsilon})<1,$ we have
$$f(p_{\epsilon})\geq
\frac{\sup_{M}f-\eta}{1-\epsilon u(p_{\epsilon})} >
\sup_{M}f-\eta,$$ completing the proof.

\end{proof}

\begin{remark}
L.J. Alias, D. Impera, and M. Rigoli  used their generalized
Omori-Yau maximum principle to obtain certain estimates
of higher order mean curvatures of hypersurfaces in some warped
product spaces, and  D. Impera \cite{DI} similarly obtained such
estimates for spacelike hypersurfaces in Lorentzian manifolds.
\end{remark}

\section{Proof of Theorem \ref{main1}}
We follow the idea of \cite{suh, cysung11}. We may choose a constant $a$
such that $f+a>0,$ because $f$ is bounded below. Let $G:M
\rightarrow \mathbb{R}^{+}$ be a $C^2$ function defined by
$G=(f+a)^{\frac{1-q}{2}}$ where $q>1$ is a constant.

Since $G$ is bounded below, Theorem \ref{main3} implies that for
any $\delta >0$ there exists a point $p_{\epsilon} \in M$ such that
\begin{equation}\label{eq3}
|\nabla G(p_{\epsilon})|<\delta,~ |\widetilde{\nabla}
G(p_{\epsilon})| <\delta,~
LG(p_{\epsilon})>-\delta, ~ \textrm{and} ~ \inf
G+\delta>G(p_{\epsilon}),
\end{equation} where $\widetilde{\nabla}$ is defined by
(\ref{eq10}). Note that $G(p_{\epsilon}) \rightarrow \inf G$ and
$f(p_{\epsilon}) \rightarrow \sup f$ as $\delta \rightarrow 0.$

By a direct calculation,
\begin{equation}\label{eq4}
\widetilde{\nabla}G|_{p_{\epsilon}}=
(\frac{1-q}{2})G(p_{\epsilon})^{\frac{q+1}{q-1}}\widetilde{\nabla}
f|_{p_{\epsilon}}.
\end{equation}
\begin{lemma}
\begin{equation}\label{eq5}
LG|_{p_{\epsilon}}=
-(\frac{q+1}{2})G(p_{\epsilon})^{\frac{2}{q-1}}\nabla
G|_{p_{\epsilon}} \cdot \widetilde{\nabla}
f|_{p_{\epsilon}}+(\frac{1-q}{2})G(p_{\epsilon})^{\frac{q+1}{q-1}}Lf|_{p_{\epsilon}}.
\end{equation}
\end{lemma}
\begin{proof}
By (\ref{eq1}), evaluating $LG$ at $p_{\epsilon}$,
we have
$$LG|_{p_{\epsilon}}
=\sum_{l}a_{ll}(p_{\epsilon})\frac{\partial^2}{\partial
x_{l}^2}G|_{p_{\epsilon}}+\sum_{l}
a_{l}(p_{\epsilon})\frac{\partial}{\partial x_{l}}G
|_{p_{\epsilon}},~ \textrm{where}~ 1 \leq l \leq n.$$ By a simple
calculation, one gets
$$\sum_{l}a_{ll}(p_{\epsilon})\frac{\partial^2}{\partial x_{l}^2}G|_{p_{\epsilon}}=-(\frac{q+1}{2})G(p_{\epsilon})^{\frac{2}{q-1}}\nabla
G|_{p_{\epsilon}} \cdot \widetilde{\nabla} f|_{p_{\epsilon}}+
\sum_{l}a_{ll}(p_{\epsilon})(\frac{1-q}{2})(f(p_{\epsilon})+a)^{\frac{-1-q}{2}}\frac{\partial^2}{\partial
x_{l}^2}f|_{p_{\epsilon}}$$ and $$\sum_{l}
a_{l}(p_{\epsilon})\frac{\partial}{\partial x_{l}}G
|_{p_{\epsilon}}=\sum_{l}
a_{l}(p_{\epsilon})(\frac{1-q}{2})(f(p_{\epsilon})+a)^{\frac{-1-q}{2}}\frac{\partial}{\partial
x_{l}}f|_{p_{\epsilon}}.$$ This yields the desired equality.
\end{proof}
By plugging (\ref{eq4}) to (\ref{eq5}), we have
$$(\frac{1-q}{2})G(p_{\epsilon})^{\frac{2q}{q-1}}Lf|_{p_{\epsilon}}=G(p_{\epsilon})LG|_{p_{\epsilon}}-(\frac{q+1}{q-1})\nabla
G(p_{\epsilon}) \cdot \widetilde{\nabla} G(p_{\epsilon}).$$ Applying
(\ref{eq3}) gives
\begin{equation}\label{eq6}
(\frac{1-q}{2})G(p_{\epsilon})^{\frac{2q}{q-1}}Lf|_{p_{\epsilon}}>
G(p_{\epsilon})(-\delta)-(\frac{q+1}{q-1})\delta^{2}.
\end{equation}
Applying $Lf \geq F(f)+ H(|\nabla f|)$ and
replacing $G$ by $(f+a)^{\frac{1-q}{2}},$ we have
\begin{equation}\label{eq7}
\frac{F(f(p_{\epsilon}))+H(|\nabla
f(p_{\epsilon})|)}{(f(p_{\epsilon})+a)^{q}}<(\frac{2\delta}{q-1})
\frac{1}{(f(p_{\epsilon})+a)^{\frac{q-1}{2}}}+\frac{2(q+1)}{(q-1)^{2}}\delta^{2}.
\end{equation}
Assume that $\sup f< \infty.$ Then as $\delta \rightarrow 0,$ since $\nabla G|_{p_{\epsilon}}\rightarrow 0,$ $G$ is bounded below by a positive constant, and
$$\nabla G|_{p_{\epsilon}}=
(\frac{1-q}{2})G(p_{\epsilon})^{\frac{q+1}{q-1}}\nabla
f|_{p_{\epsilon}},$$ we have $H(|\nabla f(p_{\epsilon})|)\rightarrow
0.$ Also, the $\mathbf{RHS}$ of (\ref{eq7}) converges to $0$ while
the $\mathbf{LHS}$ of (\ref{eq7}) converges to $\frac{F(\sup
f)}{(\sup f + a)^{q}}$ as $\delta \rightarrow 0.$ Thus, we get
$F(\sup f)\leq 0.$

Finally, it remains to show that when $\liminf_{x\rightarrow
\infty}\frac{F(x)}{x^{\nu}}>0$ for some $\nu >1,$ $f$ must be
bounded. Assume to the contrary that $\sup f = \infty.$ Then for $q< \nu,$ the
$\mathbf{RHS}$ of (\ref{eq7}) converges to $0$, while the $\mathbf{LHS}$ of (\ref{eq7}) diverges to
$\infty$ as $\delta
\rightarrow 0.$ This is a desired contradiction, which completes the proof.

\section{Proof of Theorem \ref{main2}}
We again follow the idea of \cite{suh, cysung11}. Since $-f$ is bounded below,
we can apply the proof of Theorem \ref{main1} to $-f$ with $q<1$. By
the inequality (\ref{eq6}), we get
$$(\frac{1-q}{2})G(p_{\epsilon})^{\frac{2q}{q-1}}L(-f)|_{p_{\epsilon}}>
G(p_{\epsilon})(-\delta)-\frac{|q+1|}{|q-1|}\delta^{2}.$$ Applying
$Lf \geq F(f)+ H(|\nabla f|),$ we have
$$\frac{F(f(p_{\epsilon}))+H(|\nabla f(p_{\epsilon})|)}{(-f(p_{\epsilon})+a)^{q}}\leq
\frac{L
f(p_{\epsilon})}{(-f(p_{\epsilon})+a)^{q}}<(\frac{2\delta}{1-q})
\frac{1}{(-f(p_{\epsilon})+a)^{\frac{q-1}{2}}}+\frac{2|q+1|}{(q-1)^{2}}\delta^{2}.$$
By a simple calculation,
\begin{equation}\label{eq8}
\frac{F(f(p_{\epsilon}))+H(|\nabla
f(p_{\epsilon})|)}{(-f(p_{\epsilon})+a)^{\frac{q+1}{2}}}<
\frac{2\delta}{1-q}+\frac{2|q+1|}{(q-1)^{2}}\delta^{2}(-f(p_{\epsilon})+a)^{\frac{q-1}{2}}.
\end{equation}
By the same method as above, we get $H(|\nabla
f(p_{\epsilon})|)\rightarrow 0.$ If $\inf f >- \infty,$ then $F(\inf
f) \leq 0 $ as $\delta \rightarrow 0.$

Now it only remains to show that if $\liminf_{x\rightarrow
-\infty}\frac{F(x)}{(-x)^{\nu}}>0$ for some $\nu \leq 1,$ then $f$ is
bounded. Let's assume that to the contrary $\inf f=- \infty.$ By taking $q$ such that
$\frac{q+1}{2}< \nu$ and letting $\delta \rightarrow 0$, the
$\mathbf{RHS}$ of (\ref{eq8}) converges to $0$ while the
$\mathbf{LHS}$ of (\ref{eq8}) diverges to $\infty$. This is a
contradiction completing the proof.

\section{Proof of Corollary \ref{cor3}}
Suppose that $f$ is bounded above and satisfies $L f \geq c>0$ for a
constant $c$. Applying Theorem \ref{main2} with $F=c$ and $H=0$, one conclude
that $f$ is bounded and $F(\inf f) \leq 0.$ This is
contradictory to $F\equiv c
>0.$

For a proof of Corollary \ref{cor3} $(2)$, applying Theorem
\ref{main2} with $F(f)=c|f|^{d},$ it follows that $f$ is bounded and
$c|\inf f|^{d} \leq 0$ implying $f \equiv 0.$

\section*{Acknowledgments}
The authors would like to thank Hanjin Lee for remarks about the tamed exhaustion function.
This work was supported by the National Research Foundation of Korea(NRF) grant
funded by the Korea government(MEST) (No.  2012-0000341, 2011-0002791).
\footnotesize
\bibliographystyle{amsplain}

\providecommand{\bysame}{\leavevmode\hbox
to3em{\hrulefill}\thinspace}

\end{document}